\documentclass[11pt,a4paper]{article}

\usepackage[left=3cm, top=2.5cm,bottom=2.5cm,right=3cm]{geometry}
\usepackage{mathtools,amssymb,amsthm,mathrsfs,calc,graphicx,xcolor,cleveref,dsfont,tikz,pgfplots,bm, bbm}
\usepackage{tikz-cd}

\usepackage[british]{babel}
\usepackage{amsfonts}              
\usepackage[T1]{fontenc}
\numberwithin{equation}{section}
\numberwithin{figure}{section}

\newcommand{\ind}{\mathbbm{1}}

\newtheorem {theorem}{Theorem}[section]
\newtheorem {proposition}[theorem]{Proposition}
\newtheorem {lemma}[theorem]{Lemma}
\newtheorem {corollary}[theorem]{Corollary}

{\theoremstyle{definition}

}

{
\theoremstyle{remark}
\newtheorem {remark}[theorem]{Remark}
}

\newcommand{\eqdistr}{\stackrel{d}{=}}

\def\lan{\langle}
\def\ran{\rangle}

\def\EE{\mathbb{E}}

\def\NN{\mathbb{N}}

\def\PP{\mathbb{P}}

\def\RR{\mathbb{R}}
\def\SS{\mathbb{S}}

\def\cC{\mathcal{C}}

\def\cH{\mathcal{H}}

\def\cK{\mathcal{K}}

\def\cP{\mathcal{P}}

\def\dint{\textup{d}}
\def\Tan{\textup{Tan}}

\def\pos{\textup{pos}}
\def\PC{\textup{PC}}
\def\vol{\textup{vol}}
\def\cones{\textup{Cns}}

\begin{document}

\title{\bfseries A new approach to weak convergence\\ of random cones and polytopes}

\author{Zakhar Kabluchko\footnotemark[1],\; Daniel Temesvari\footnotemark[2]\;\; and Christoph Th\"ale\footnotemark[3]}

\date{}
\renewcommand{\thefootnote}{\fnsymbol{footnote}}
\footnotetext[1]{Institut f\"ur Mathematische Stochastik, Westf\"alische Wilhelms-Universit\"at M\"unster, Germany. Email: zakhar.kabluchko@uni-muenster.de}

\footnotetext[2]{
Institut f\"ur Diskrete Mathematik und Geometrie,
Technische Universit\"at Wien,
Austria.
Email: daniel.temesvari@tuwien.ac.at}

\footnotetext[3]{Fakult\"at f\"ur Mathematik, Ruhr-Universit\"at Bochum, Germany. Email: christoph.thaele@rub.de}

\maketitle

\begin{abstract}
\noindent  A new approach to prove weak convergence of random polytopes on the space of compact convex sets is presented. This is used to show that the profile of the rescaled  Schl\"afli random cone of a random conical tessellation generated by $n$ independent and uniformly distributed random linear hyperplanes in $\mathbb{R}^{d+1}$ weakly converges to the typical cell of a stationary and isotropic Poisson hyperplane tessellation in $\mathbb{R}^d$, as $n\to\infty$.
\end{abstract}

\section{Introduction and description of the main result}

\subsection{Introduction}
Random polytopes are one of the central objects studied in convex and stochastic geometry. In the most common model, a random polytope arises as a convex hull of a number of independent random points chosen, for example, with respect to the standard Gaussian distribution or with respect to the uniform distribution in a prescribed convex body. Starting with the seminal work of R\'enyi and Sulanke~\cite{renyi_sulanke1,renyi_sulanke2}, exact and asymptotic descriptions have been found for many characteristics of random polytopes. Examples include the expected intrinsic volumes or the expected face numbers. For these functionals also variance asymptotics and accompanying central limit theorems are available; see \cite{Baranyth,Baranyvu,ReitznerCLT,TTW}. We refer the reader to the survey articles~\cite{BaranySurvey,HugSurvey,ReitznerSurvey} and the book of Schneider and Weil~\cite{SW} for more background material, information and references.

Much less is known about random polytopes in other spaces of constant curvature. In this paper, we shall be interested in random spherical polytopes in the $d$-dimensional unit sphere $\SS^d$. These are defined as intersections of $\SS^d$ with random polyhedral cones in $\RR^{d+1}$ which, in turn, are intersections of finitely many closed random half-spaces whose boundaries pass through the origin. In other words, a polyhedral cone is just a set of solutions to a finite system of linear homogeneous inequalities. The study of random polyhedral cones has been initiated by Cover and Efron~\cite{CoverEfron} and continued by Hug and Schneider~\cite{HugSchneiderConicalTessellations}  {and Schneider \cite{SchneiderKinematicCones}}. Random spherical tessellations, i.e.\ random decompositions of the sphere into finitely many spherical polytopes with disjoint interiors, were studied in the works of Miles~\cite{MilesSphere} and Arbeiter and Z\"ahle~\cite{ArbeiterZaehle}; see also~\cite{HugThaele} and~\cite{KabluchkoThaele_VoronoiSphere} for more recent contributions.
The recent renewed interest in spherical convex geometry is due to the fact that geometric properties of polyhedral cones have found striking applications in convex optimization and compressed sensing; see~\cite{AmelunxenLotz,AmelunxenBuergisser,AmelunxenLotzMcCoyTropp,GoldsteinNourdinPeccati,McCoyTropp}.
In addition, random polytopes on the sphere show new phenomena, which do not have analogues in Euclidean spaces. A first such phenomenon has recently been described by B\'ar\'any, Hug, Reitzner and Schneider \cite{BaranyHugReitznerSchneider} for random convex hulls in the half-sphere and was more systematically investigated in~\cite{MarynychKabluchkoTemesvariThaele} and~\cite{kabluchko_poisson_zero}.

\subsection{Definition of random cones}

We shall be interested in random polyhedral cones defined as follows.  Consider $n\in\NN$ random  unit vectors $U_1,\ldots,U_n$ sampled uniformly and independently from the $d$-dimensional unit sphere $\SS^d$ in  $\RR^{d+1}$. The orthogonal complements of these vectors, denoted by $H_1:= U_1^\bot, \ldots, H_n := U_n^\bot$, are $n$ random hyperplanes in $\RR^{d+1}$ passing through the origin {, which are distributed according to the rotation-invariant Haar probability measure on the Grassmannian of all $d$-dimensional linear subspaces of $\RR^{d+1}$}. These hyperplanes dissect the space into a finite number of polyhedral cones. By a classical result of Steiner and Schl\"afli~\cite[Lemma~8.2.1]{SW}, the number of these cones is almost surely constant and equals
\begin{equation}\label{eq:C_n_d+1}
C(n,d+1) := 2\sum_{m=0}^{d}{n-1\choose m}.
\end{equation}
The \textbf{Schl\"afli random cone} $S_n$ is a polyhedral cone selected uniformly at random from this collection of cones {; see \cite{HugSchneiderConicalTessellations,SchneiderKinematicCones}}. By definition, each cone has the same probability of $1/C(n,d+1)$ to be selected.
An equivalent way to think of the Schl\"afli cone is as follows. Consider the random cone
$$
D_n:= \pos (U_1,\ldots,U_n) := \{\lambda_1 U_1 + \ldots +\lambda_n U_n: \lambda_1,\ldots,\lambda_n\geq 0\},
$$
where $U_1,\ldots,U_n$ are independent and uniformly distributed random points on $\SS^d$ as before.
The \textbf{polar} (or dual) \textbf{cone} of $D_n$ is defined by
$$
D_n^\circ := \{y\in\RR^{d+1}:\lan x,y\ran\leq 0\ \text{ for all } x\in D_n\},
$$
where $\langle\,\cdot\,,\, \cdot\,\rangle$ denotes the standard Euclidean scalar product. In fact, $D_n^\circ$ can explicitly be given as
\begin{equation}\label{eq:system}
D_n^\circ = \{y \in \RR^{d+1}: \langle U_1,y\rangle \leq 0,\ldots, \langle U_n, y\rangle\leq 0\}.
\end{equation}
Thus, $D_n^\circ$ is just the set of solutions to a system of  $n$ random linear homogeneous inequalities in the $d+1$ unknowns $y = (y_1,\ldots, y_{d+1})$. It can happen that the only solution is the trivial solution $y=0$, which occurs if and only if $D_n = \RR^{d+1}$. According to a theorem of Wendel~\cite{Wendel}, see also~\cite[Theorem 8.2.1]{SW}, the probability of this event equals
$$
\PP[D_n^\circ = \{0\}] = \PP[D_n = \RR^{d+1}] = 1-  {C(n,d+1)\over 2^n}.
$$
It turns out that the conditional distribution of $D_n^\circ$ on the event $\{D_n^\circ\neq \{0\}\}$ coincides with the distribution of the Schl\"afli cone $S_n$; see ~\cite[Theorem~3.1]{HugSchneiderConicalTessellations}. Thus, the Schl\"afli cone is just the set of solutions to a system of random linear inequalities given that there is at least one non-zero solution.

Schl\"afli random cones $S_n$ were introduced and studied by Cover and Efron~\cite{CoverEfron} and Hug and Schneider~\cite{HugSchneiderConicalTessellations} along with their polars $S_n^\circ$, called the \textbf{Cover-Efron random cones}. These authors determined expected values of several natural characteristics of $S_n$ and $S_n^\circ$. For example, Cover and Efron~\cite[Theorems~1' and~3']{CoverEfron} calculated explicitly the expectation of $f_k(S_n\cap \SS^{d})$, the number of $k$-dimensional faces of the spherical polytope $S_n\cap \SS^{d}$:
\begin{align*}
 {\EE f_k(S_n\cap\SS^d) = {2^{d-k}{n\choose d-k}C(n-d+k,k+1)\over C(n,d+1)},\qquad k\in\{0,1,\ldots,d\}.}
\end{align*}
Passing to the large $n$ limit  {and using that $C(n,d+1)$ is asymptotic to ${2\over d!}n^{d}$ and $C(n-d+k,k+1)$ to ${2\over k!}n^{k}$, as $n\to\infty$},  they derived the formula
\begin{equation}\label{eq:cover_efron_cube}
\begin{split}
\lim_{n\to\infty} \EE f_k(S_n\cap \SS^{d}) = 2^{d-k} \binom{d}{k}, \qquad k \in \{0,1,\ldots, d\}.
\end{split}
\end{equation}
The starting point of the present work was the observation, due to Cover and Efron~\cite[Section~4]{CoverEfron}, that the number on the right-hand side coincides with  the number of $k$-dimensional faces of the $d$-dimensional cube. In the last sentence of their paper, Cover and Efron~\cite{CoverEfron} wrote:  ``Loosely speaking, the ``expected'' cross section of $W$ is a cube''. This is certainly true with regard to the number of faces, but does it mean that the \textit{shape} of the spherical polytope $S_n\cap \SS^d$ approaches the shape of the cube in the large $n$ limit? As we shall show, the answer is negative. Our main result states that, in a suitable sense,  $S_n\cap \SS^d$ looks like a typical cell $Z$ of a stationary and isotropic Poisson hyperplane tessellation in $\RR^d$ {, which is another well studied object in stochastic geometry}. At the moment, the reader may think of $Z$ as of a uniformly chosen cell from all cells of the Poisson hyperplane tessellation within a ``large'' observational window;  {Figure \ref{fig2} shows two simulations of Poisson hyperplane tessellation in $\RR^2$}. A rigorous definition will be given in Section~\ref{subsec:weak_conv_main_proof} below.  Coincidentally, $Z$ has the same expected  number of $k$-dimensional faces as the cube in every dimension $k\in \{0,1,\ldots, d\}$:
$$
 {\EE f_k(Z) = 2^{d-k}{d\choose k};}
$$
see~\cite[Theorem~10.3.1]{SW}. This gives an explanation of~\eqref{eq:cover_efron_cube} on the level of a distributional limit theorem.

\subsection{Main result}

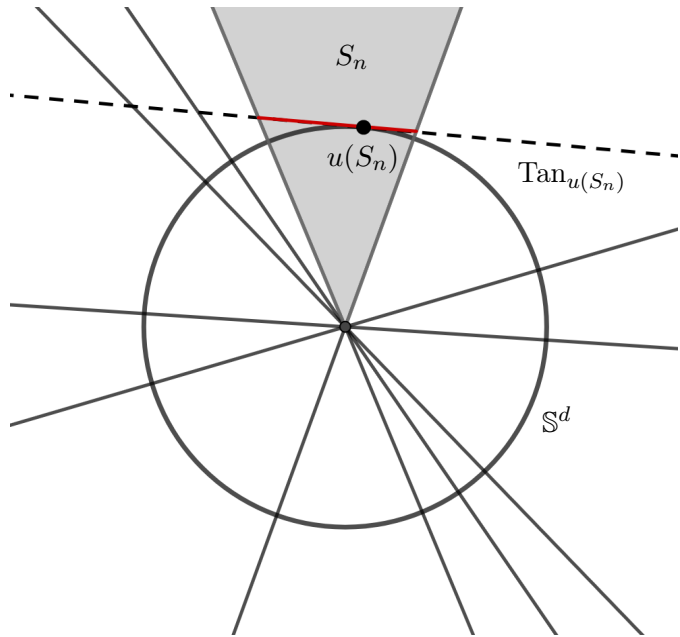
\begin{figure}[t]
\begin{center}
	\begin{tikzpicture}
	\pgftext{\includegraphics[width=250pt]{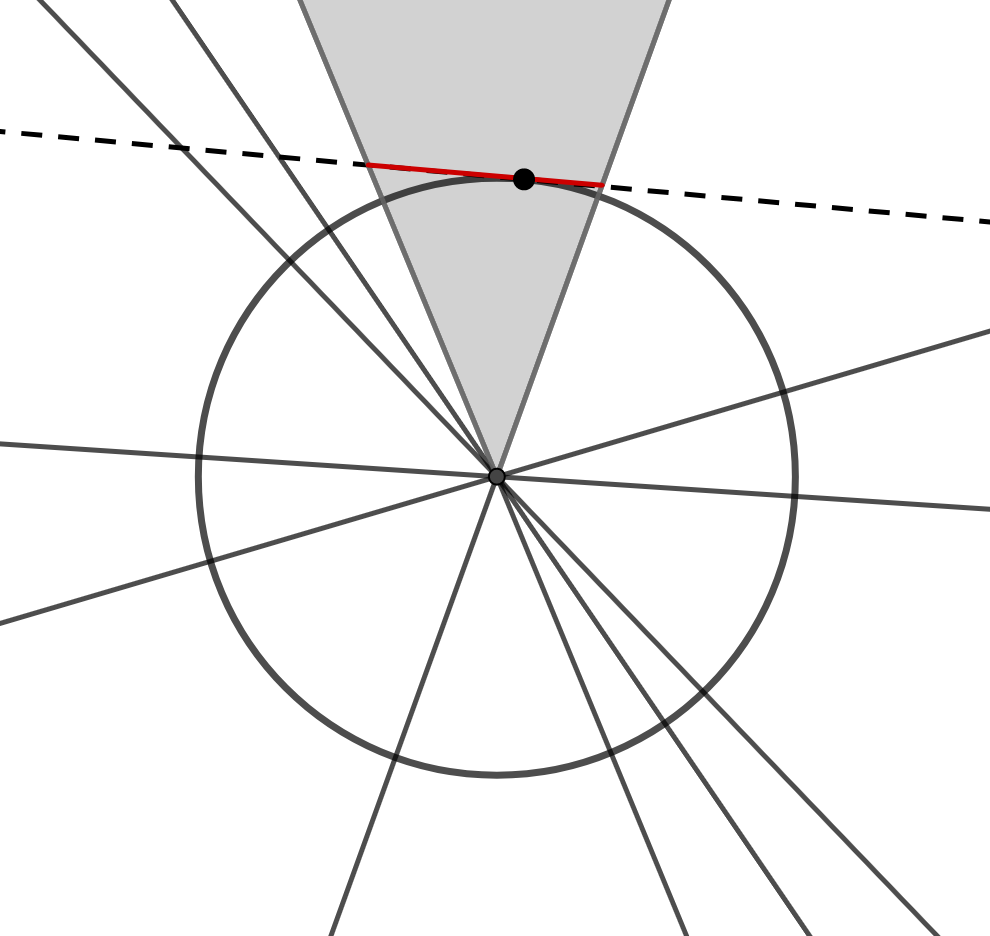}} at (0pt,0pt);
	\node at (2.8,-1.3) {$\SS^d$};
	\node at (0.1,3.5) {$S_n$};
	\node at (0.25,2.15) {$u(S_n)$};
	\node at (3,1.9) {$\Tan_{u(S_n)}$};
	\end{tikzpicture}
\end{center}
\caption{A conical random tessellation with the Schl\"afli random cone $S_n$ and its profile (in red) in the tangent space $\Tan_{u(S_n)}$ at a uniform random point $u(S_n)$.}
\label{fig1}
\end{figure}

Our main result can roughly be described as follows (all details will be provided in the remaining parts of this paper). Let $S_n$ be a Schl\"afli random cone  {as defined in the previous section}. In the large $n$ limit, the cone $S_n$ becomes ``thin'' (close to a ray) and therefore needs to be rescaled (enlarged) to have a non-trivial weak limit. The rescaling is done as follows. Take a random point $u(S_n)$  distributed uniformly in the spherical polytope $S_n\cap \SS^d$ and let $\Tan_{u(S_n)}$ be the tangent space of $\SS^d$ at $u(S_n)$. The intersection $S_n\cap\Tan_{u(S_n)}$ of the Schl\"afli cone with this tangent space is what we call the \textbf{profile} of $S_n$, see Figure \ref{fig1}. After identifying $\Tan_{u(S_n)}$ with $\RR^d$ using some isometry that maps $u(S_n)$ to the origin of $\RR^d$, the profile can be regarded as a random (possibly unbounded) convex set in $\RR^d$.

\begin{theorem}\label{thm:Intro}
As $n\to\infty$, the profile of the  Schl\"afli random cone $S_n$, multiplied by the factor $n$,  weakly converges to the typical cell $Z$ of a stationary and isotropic Poisson hyperplane tessellation in $\RR^d$ with intensity
\begin{equation}\label{eq:gamma}
\gamma:={1\over\sqrt{\pi}}{\Gamma({d+1\over 2})\over\Gamma({d\over 2})}.
\end{equation}
The typical cell $Z$ is centered at a random point distributed uniformly inside the cell. The convergence takes place on the space $\cK^d$ of compact convex subsets of $\RR^d$ to which the profile belongs with probability converging to $1$.
\end{theorem}

On the intuitive level, Theorem~\ref{thm:Intro} can be understood as follows.  As the number $n$ of hyperplanes  grows, the cells of the spherical tessellation generated by these hyperplanes become small. Since on small scales the sphere looks approximately like a flat $d$-dimensional Euclidean space, it is not surprising that a spherical cell chosen uniformly at random from the collection of all available spherical cells looks approximately like a similarly defined object in the Euclidean setting.
However, considerable difficulties arise when trying to make this idea rigorous. The main obstacle is the lack of a convenient representation for the distribution of the Schl\"afli cone. Its definition is not well-suited to study weak convergence and the only explicit representation we are aware of will be given in Lemma~\ref{lemma:size_biased}. It states that the Schl\"afli cone can be  obtained from another natural random cone after biasing it by the inverse of the solid angle. However, since the inverse angle is not bounded, standard weak convergence theory does not directly apply to this representation.  In order to prove Theorem \ref{thm:Intro} we develop a new general method to prove weak convergence of random polytopes in $\RR^d$,  {which is of interest in its own right}.

The idea behind this new method is to assign to each polytope with $m \geq d+1$ vertices in $\RR^d$ its ``coordinate representation'' obtained by listing the coordinates of the vertices in some fixed order.  The space of the polytopes can then essentially be identified with some subset of the disjoint union of all the spaces $(\RR^{d})^m$, $m\geq d+1$. There is a natural analogue of the Lebegue measure on this subset, and the distributions of many natural random polytopes possess well-defined densities with respect to this measure.  If, for a sequence of random polytopes $P_1,P_2,\ldots$ with respective densities $f_1,f_2,\ldots$, we can prove the pointwise convergence of the densities to the density $f$ of some random polytope $P$,  then by a Scheff\'e-type result we can conclude that the sequence of random polytopes $P_1,P_2,\ldots$ converges weakly to the random polytope $P$ on the space of compact convex bodies equipped with the Hausdorff metric.

\medbreak

The remaining parts of this text are structured as follows. In Section~\ref{sec:WeakConvergence} we present our new approach to weak convergence of random polytopes. In Section~\ref{sec:cones_profiles} we formally introduce the  Schl\"afli random cones (together with several other models of random cones). Finally, in Section~\ref{sec:ConesAndProof} we give a proof of Theorem \ref{thm:Intro}.

\section{Weak convergence of random polytopes}\label{sec:WeakConvergence}\label{sec:weak_conv}

 {The purpose of this section is to describe our new method proving weak convergence of random polytopes. We start by introducing the necessary (topological) spaces of polytopes and and then state the new Scheff\'e-type condition for weak convergence.}

\subsection{Spaces of polytopes}
Fix a dimension $d\geq 1$. By $\cK^d$ we denote the space of non-empty  compact convex subsets in $\RR^d$ and supply $\cK^d$ with the topology $\tau_H^d$ generated by the Hausdorff distance $d_H$. We recall that
$$
d_H(K,L) = \min\{\varepsilon\geq 0:K\subseteq L+\varepsilon B^d,L\subseteq K+\varepsilon B^d\},\qquad K,L\in\cK^d,
$$
where $B^d$ stands for the centred unit ball in $\RR^d$ . The set of polytopes in $\RR^d$ is denoted by $\cP^d\subset\cK^d$.
A random element, defined on some probability space and taking values in $\cP^d$, is called a random polytope.

For all integers $m\geq d+1$ we let $\cP_m^d\subset \cP^d$ be the set of polytopes in $\RR^d$ with exactly $m$ vertices and such that the first coordinates of these vertices are pairwise different. By the latter condition we mean the following. If $p\in\cP^d$ is a polytope with $m\geq d+1$ vertices $x_i=(x_{i,1},\ldots,x_{i,d})$, $i\in\{1,\ldots,m\}$, then $p\in\cP_m^d$ if and only if $x_{1,1}<x_{2,1}<\ldots<x_{m,1}$ (possibly after reordering of the vertices). We supply $\cP_m^d$ with the topology induced from its embedding into $\cK^d$.
Furthermore, for $m\geq d+1$, we define the open sets
\begin{multline*}
\widetilde \cP_m^d =
\{
(x_1,\ldots,x_m)\in(\RR^d)^m:
x=(x_{i,1},\ldots,x_{i,d})\in\RR^d,
\\
x_{1,1}<x_{2,1}<\ldots<x_{m,1},
\;\;\;
x_1,\ldots, x_m \text{ are in convex position}
\}.
\end{multline*}
We supply $\widetilde \cP_m^d$ with the restriction $\widetilde \tau_m^d$ of the standard Euclidean topology in $(\RR^d)^m$. The sets  (but not the topological spaces) $\widetilde \cP_m^d$ and $\cP_m^d$  can be identified via the bijective maps
\begin{align*}
	\iota_m: \widetilde \cP_m^d \to \cP_m^d,\quad(x_1,\ldots,x_m) \mapsto \mathrm{conv}(\{x_1,\ldots,x_m\}).
\end{align*}
We also let
$$
\widetilde \cP_\infty^d:=\coprod\limits_{m=d+1}^\infty \widetilde \cP_m^d
$$
be the disjoint union of the spaces $\widetilde \cP_m^d$, $m\geq d+1$, and supply $\widetilde \cP_\infty^d$ with the disjoint union topology $\widetilde \tau_\infty^d:=\coprod\limits_{m=d+1}^\infty\widetilde \tau_m^d$. We recall that the topology $\widetilde \tau_\infty^d$ is the finest topology on $\widetilde \cP_\infty^d$ such that all canonical injections $\widetilde \cP_m^d\hookrightarrow \widetilde \cP_\infty^d$, $m\geq d+1$, are continuous.

\begin{lemma}
The topological space $(\widetilde \cP_\infty^d,\widetilde \tau_\infty^d)$ is a Polish space.
\end{lemma}
\begin{proof}
Each of the spaces $\widetilde \cP_m^d$, $m\geq d+1$, is Polish as an open subset of $(\RR^d)^m$; see, e.g.~\cite[Proposition~8.1.4]{Cohn}. Thus, as a countable disjoint union of Polish spaces, $\widetilde \cP_\infty^d$ is a Polish space too; see \cite[Proposition~8.1.2]{Cohn}.
\end{proof}

Let now $\cP_\infty^d$ be the  union of the disjoint sets $\cP_m^d$, $m\geq d+1$, i.e.,
$$
\cP_\infty^d:=\bigcup_{m=d+1}^\infty \cP_m^d \; \subset \cK^d.
$$
This allows us to introduce the bijective map
\begin{align*}
\iota: \widetilde \cP_\infty^d \to \cP_\infty^d \subset \cK^d,\quad(x_1,\ldots,x_m) \mapsto \mathrm{conv}(\{x_1,\ldots,x_m\}).
\end{align*}
As stated in the last claim of~\cite[Theorem 12.3.5]{SW}, the restriction of this map to $\widetilde \cP_m^d$ for each fixed $m\geq d+1$ is continuous. Since   $\widetilde \cP_\infty^d$ is defined as the disjoint union of these topological spaces, the map $\iota$ is continuous.   The diagram in Figure \ref{fig:diag} illustrates the underlying structural associations of the spaces and maps we have introduced so far.

\begin{figure}[t]\label{fig:diag}
\[
\begin{tikzcd}
\widetilde{\mathcal P}_\infty^d \arrow[r,"\iota"]& \mathcal{P}_\infty^d \arrow[r, hook] & \mathcal{P}^d \arrow[r, hook] & \mathcal{K}^d \\
\widetilde{\mathcal P}_m^d \arrow[u, hook] \arrow[r, "\iota_m"] & \mathcal{P}_m^d \arrow[u, hook]      &                               &
\end{tikzcd}
\]	
\caption{Spaces of polytopes and their structural associations.}
\end{figure}

\subsection{Scheff\'e-type condition for weak convergence of random polytopes}
For a topological space $(X,\tau)$ we denote by $\cC_b(X,\tau)$ the set of functions $h:X\to\RR$ which are bounded and continuous with respect to $\tau$ and the standard Euclidean topology on $\RR$. The restriction of $h$ to a subset $A\subseteq X$ is denoted by $h|_A$. Let us also recall that a sequence of probability measures $\mu_n$, $n\geq 1$, on $(X,\tau)$ weakly converges to another probability measure $\mu$ on $(X,\tau)$ provided that $\int_X h\,\dint\mu_n\to\int_X h\,\dint\mu$, as $n\to\infty$, for all $h\in\cC_b(X,\tau)$.

If $\mu$ is a measure on $\widetilde \cP_\infty^d$, we denote by $\mu^\iota := \mu \circ \iota^{-1}$ its image measure under the continuous map $\iota:\widetilde \cP_\infty^d \to \cP_\infty^d$. Thus, $\mu^\iota $ is a measure on $\cP_\infty^d$ (which is a subset of $\cK^d$).
The next proposition states that in order to prove weak convergence of random polytopes it is enough to check the weak convergence of their ``coordinate representations'' in $\widetilde \cP_\infty^d$, which is usually easier.

\begin{proposition}\label{prop:WeakConvergenceAbstract}
Let $(\mu_n)_{n\geq 1}$ be a sequence of probability measures on $\widetilde \cP_\infty^d$. Suppose that $(\mu_n)_{n\geq 1}$ weakly converges on $(\widetilde \cP_\infty^d,\widetilde \tau_\infty^d)$ to some probability measure $\mu$. Then $(\mu_n^\iota)_{n\geq 1}$ converges weakly to $\mu^\iota$ on $(\cK^d,\tau_H^d)$.
\end{proposition}

\begin{proof}
This is just an application of the continuous mapping theorem~\cite[p.~20]{billingsley} to the map $\iota$.
\end{proof}

We are now going to state a result which is our main new device to prove weak convergence of random polytopes and which will be one of the key tools in the proof of Theorem~\ref{thm:Intro}.
In what follows we denote by $\mu_\infty^d$ a measure on $\widetilde \cP_\infty^d$ with the property that for each $m\geq d+1$ the restriction  of $\mu_\infty^d$ to $\widetilde \cP_m^d$ coincides with the Lebesgue measure on $\widetilde \cP_m^d\subset (\RR^d)^m$. One might think of $\mu_\infty^d$ as a Lebesgue measure on the space $\widetilde \cP_\infty^d$.   We introduce the notation $p^\iota:=\iota^{-1}(p)\in \widetilde \cP_\infty^d$ for the preimage of a polytope $p \in \cP_\infty^d$ under the bijective map $\iota$. We may view $p^\iota$ as the ``coordinate representation'' of the polytope $p$.  The next result is a Scheff\'e-type sufficient condition for the weak convergence of random polytopes.

\begin{proposition}\label{prop:WeakCovAbstract}
Let $(T_n)_{n\geq 1}$ be a sequence of random polytopes in $\RR^d$ and $T$ be a random polytope in $\RR^d$ with the property that $\PP[T_n\in\cP_\infty^d]=\PP[T\in \cP_\infty^d]=1$ for all $n\geq 1$. Assume that for each $n\geq 1$ the distribution $\mu_{T^\iota_n}$ of $T^\iota_n$ has a density $\varphi_n={\dint\mu_{T^\iota_n}\over\dint\mu_\infty^d}$ and that the distribution $\mu_{T^\iota}$ of $T^\iota$ has a density $\varphi={\dint\mu_{T^\iota}\over\dint\mu_\infty^d}$ with respect to the measure $\mu_\infty^d$. Suppose that $\varphi_n\to\varphi$ pointwise on $\widetilde \cP_\infty^d$, as $n\to\infty$. Then, $T_n\to T$  weakly, as $n\to\infty$, on $(\cK^d,\tau_H^d)$.
\end{proposition}
\begin{proof}
We prove the weak convergence of $(\mu_{T^\iota_n})_{n\geq 1}$ to $\mu_{T^\iota}$, as $n\to\infty$, on $(\widetilde \cP_\infty^d,\widetilde \tau_\infty^d)$, which in view of Proposition \ref{prop:WeakConvergenceAbstract} yields the claim. Let $h\in\cC_b(\widetilde \cP_\infty^d,\widetilde \tau_\infty^d)$ and observe that
\begin{align*}
\Bigg|\int_{\widetilde \cP_\infty^d} h\,\dint\mu_{T^\iota_n}-\int_{\widetilde \cP_\infty^d} h\,\dint\mu_{T^\iota}\Bigg| &= \Bigg|\int_{\widetilde \cP_\infty^d} h\cdot\varphi_n\,\dint\mu_\infty^d-\int_{\widetilde \cP_\infty^d} h\cdot\varphi\,\dint\mu_\infty^d\Bigg|\\
&\leq \|h\|_\infty\int_{\widetilde \cP_\infty^d} |\varphi_n-\varphi|\,\dint\mu_\infty^d,
\end{align*}
where $\|h\|_\infty=\sup\{|h(p)|:p\in\widetilde\cP_\infty^d\}$. Since $h$ is bounded, $\|h\|_\infty<\infty$. Moreover, by assumption we have that $\varphi_n\to\varphi$ pointwise on $\widetilde \cP_\infty^d$ {, as $n\to\infty$}. By Scheff\'e's lemma (see \cite[Section~3.2.1]{Durrett}) this implies that $\int_{\widetilde \cP_\infty^d} |\varphi_n-\varphi|\,\dint\mu_\infty^d\to 0$, as $n\to\infty$. Thus,
\begin{align*}
\lim_{n\to\infty}\Bigg|\int_{\widetilde \cP_\infty^d} h\,\dint\mu_{T^\iota_n}-\int_{\widetilde \cP_\infty^d} h\,\dint\mu_{T^\iota}\Bigg| &= 0,
\end{align*}
which is the desired weak convergence of $(\mu_{T^\iota_n})_{n\geq 1}$ to $\mu_{T^\iota}$ on $(\widetilde \cP_\infty^d,\widetilde \tau_\infty^d)$.
\end{proof}

\section{Random cones and their profiles}\label{sec:cones_profiles}

In this section we  describe the construction of four types of random cones we are interested in. These include the Schl\"afli cones, their polar versions called the Cover-Efron cones (to borrow the terminology introduced in~\cite{HugSchneiderConicalTessellations}), as well as the half-space versions of these two types.

\subsection{Cover-Efron  and Schl\"afli cones}

\begin{figure}[t]
	\begin{center}
	\begin{tikzpicture}
\pgftext{\includegraphics[width=250pt]{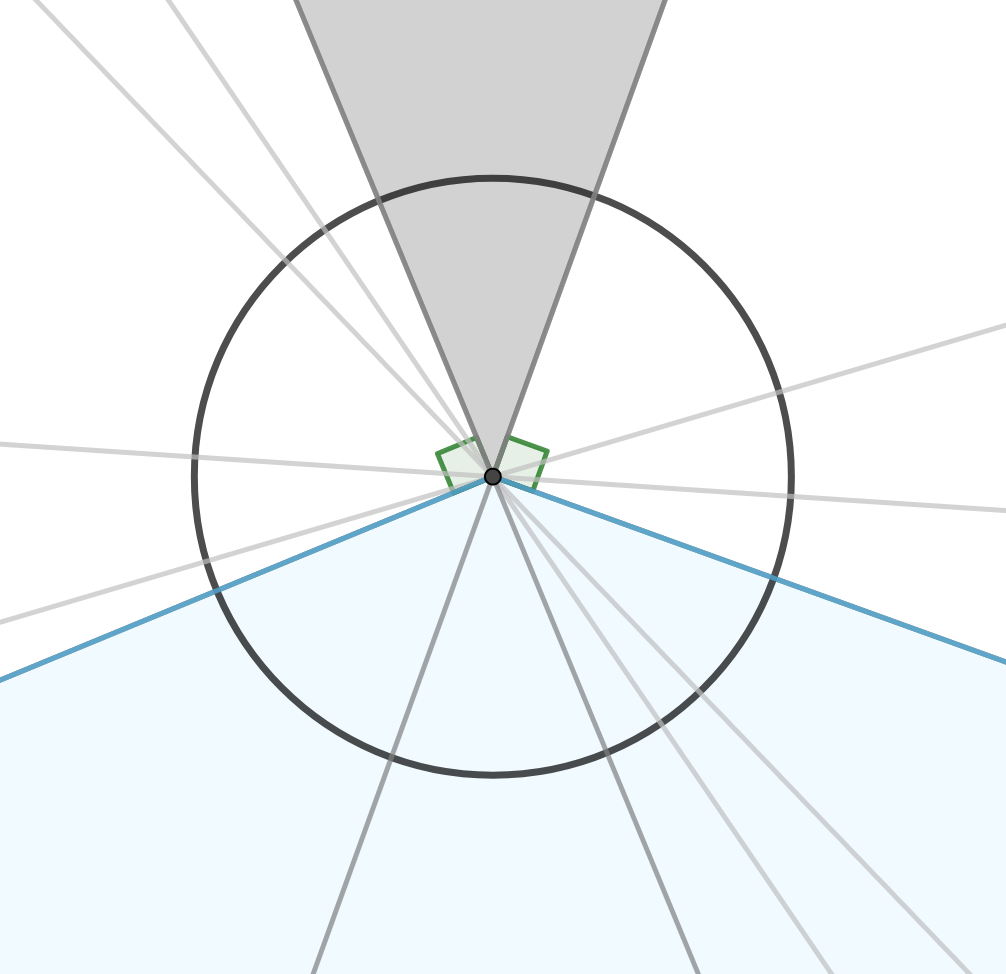}} at (0pt,0pt);
\node at (2.6,1.5) {$\SS^d$};
\node at (0,3.5) {$S_n$};
\node at (0,-3.5) {$S_n^\circ$};
\end{tikzpicture}
	\end{center}
	\caption{A conical random tessellation with the Schl\"afli random cone $S_n$ and its polar $S_n^\circ$, which has the same distribution as the Cover-Efron random cone.}
	\label{fig3}
\end{figure}

To start with, fix an integer  $d\geq 1$ and consider the $d$-dimensional unit sphere $\SS^d$ in $\RR^{d+1}$. The normalized spherical Lebesgue measure on $\SS^d$ is denoted by $\sigma_d$. For $n\in\NN$, let $U_1,\ldots,U_n$ be independent and uniformly distributed random points on $\SS^d$.  We define the random polyhedral cone $D_n$ by
$$
D_n := \pos(U_1,\ldots,U_n).
$$
According to Wendel's theorem, see~\cite{Wendel} or~\cite[Theorem 8.2.1]{SW}, we have that
$$
p_n^{(d)} := \PP[D_n\neq\RR^{d+1}] = {C(n,d+1)\over 2^n},
$$
where we recall the notation
\begin{align*}
C(r,s) := 2\sum_{m=0}^{s-1}{r-1\choose m},\qquad r,s\in\{1,2,\ldots\}.
\end{align*}
Following \cite{HugSchneiderConicalTessellations}, the \textbf{Cover-Efron random cone} is defined as a random polyhedral cone in $\RR^{d+1}$ whose distribution is the conditional distribution of $D_n$ conditioned on the event  $\{D_n\neq\RR^{d+1}\}$. Formally, its distribution is given by $\mu_{D_n\,|\,\{D_n\neq\RR^{d+1}\}}$, where we write $\mu_{(\,\cdot\,)}$ for the law of a random element and use this notation for conditional laws as well.

To  define the Schl\"afli random cones, let $\cH$ denote the space of hyperplanes in $\RR^{d+1}$ passing through the origin. The space $\cH$ carries a  {unique} rotation-invariant probability measure $\nu_d$ defined as the image of $\sigma_d$ under the map $\perp:\SS^d\to\cH,x\mapsto x^\perp$.  Now, for $n\in\NN$, let $H_1,\ldots,H_n$ be independent random hyperplanes with distribution $\nu_d$. With probability $1$, these hyperplanes partition $\RR^{d+1}$ into the constant number $C(n,d+1)$ of random closed polyhedral cones; see \cite[Lemma 8.2.1]{SW}. We denote by $\cones(H_1,\ldots,H_n)$ the collection of these cones. Following \cite{HugSchneiderConicalTessellations,SchneiderKinematicCones}, the  \textbf{Schl\"afli random cone} $S_n$ is a random closed cone picked uniformly at random from $\cones(H_1,\ldots,H_n)$. More precisely,
\begin{equation}\label{eq:schlaefli_def}
\PP[S_n\in\,\cdot\,] = {1\over C(n,d+1)} \int_{\cH^n}\sum_{C\in \cones(H_1,\ldots,H_n)}\ind\{C\in\,\cdot\,\}\,\nu_d^n(\dint(H_1,\ldots,H_n)).
\end{equation}
The  Schl\"afli random cone is related to the Cover-Efron random cone by the concept of conical polarity  {(or duality)}. Namely, let $C^\circ=\{y\in\RR^{d+1}:\lan x,y\ran\leq 0\text{ for all }x\in C\}$ denote the polar (or dual) cone of a cone $C\subset \RR^{d+1}$. Then,  it was shown in \cite[Theorem 3.1]{HugSchneiderConicalTessellations} that $D_n$ and $S_n$ are related by  {the distributional identity}
\begin{align}\label{eq:DualD_nAndS_n}
\mu_{D_n\,|\, \{D_n\neq\RR^{d+1}\}} = \mu_{S_n^\circ}.
\end{align}
That is, the Cover-Efron random cone has the same distribution as the polar of the Schl\"afli random cone (and vice versa because $C^{\circ\circ} = C$) {; see Figure \ref{fig3} for an illustration.}

\subsection{Random cones on the half-sphere}
Next, we are going to define the half-spherical versions of the random cones introduced above.  Fix an arbitrary point $e\in\SS^d$ (for concreteness, the north pole of $\SS^d$) and define the upper half-sphere $\SS_e^d:=\{x\in\SS^d:\lan x,e\ran\geq 0\}$.
Let $X_1,\ldots,X_n$ be independent random points having the uniform distribution on $\SS_e^d$. The polyhedral random  cones
\begin{equation}\label{eq:R_n_def}
R_n:=\pos(X_1,\ldots,X_n)
\end{equation}
were intensively studied in \cite{BaranyHugReitznerSchneider,kabluchko_poisson_zero,MarynychKabluchkoTemesvariThaele}. For example, explicit formulae for the expected number of $k$-dimensional faces and for the expected solid angle of $R_n$ are available; see~\cite{kabluchko_poisson_zero}.
To describe the polar cone of $R_n$, let $H_1,\ldots,H_n$ be the linear hyperplanes orthogonal to $X_1,\ldots,X_n$. Consistent with the notation used above, $H_1,\ldots,H_n$ are independently distributed according to the uniform probability measure $\nu_d$  {on $\cH$}. Let $S_n^{-e}$ be the almost surely uniquely determined cone from $\cones(H_1,\ldots,H_n)$ containing the south pole $-e$.  Then, by definition of the polar cone, we have
\begin{align}\label{eq:DualityRnAndSne}
\mu_{R_n^\circ} = \mu_{S_n^{-e}}.
\end{align}
Note that $R_n$ is contained in the upper half-space, but $S_n^{-e}$ need not be contained in the lower half-space. The dual statement is that $S_n^{-e}$ contains $-e$, but $R_n$ need not contain $e$.

Next we are going to state a relation between $S_n^{-e}$ and the Schl\"afli random cone $S_n$. For a full-dimensional cone $C\subset\RR^{d+1}$ the solid angle $\alpha(C)$ is given by $\alpha(C) := \sigma_d(C\cap \SS^d)$. Given a realization of the Schl\"afli cone $S_n$, denote by $u(S_n)$ a random point sampled uniformly from $S_n\cap \SS^d$. The next lemma states  essentially that if we rotate the cone $S_n$ so that $u(S_n)$ becomes the south pole $-e$, then the resulting cone has the same law as $S_n^{-e}$  up to biasing by the solid angle $\alpha(S_n)$. To make this precise, fix for every point $v\in \SS^d$ some orthogonal transformation $O_v\in \text{SO}(d+1)$ such that $O_v v = -e$ and such that the function $(v_1,v_2)\mapsto O_{v_1} v_2$  is Borel-measurable.

\begin{lemma}\label{lemma:size_biased}
For every $n\in\NN$ and every  non-negative Borel {-measurable} function $f$ on the space $\mathcal K_{\text{con}}$ of closed convex cones in $\RR^{d+1}$ we have
$$
\EE [f( O_{u(S_n)} S_n) C(n,d+1) \alpha(S_n)] = \EE [f(S_n^{-e})].
$$
\end{lemma}

\begin{remark}
In other words,  {the statement of Lemma \ref{lemma:size_biased}} means that the random cone $S_n^{-e}$ is a size-biased version of the random cone $O_{u(S_n)} S_n$, where the size is measured in terms of the solid angle.
A version of this lemma in which no reference point $u(S_n)$ is chosen and instead the equality is stated for rotationally invariant $f$ only, was obtained in~\cite[Lemma 5.2]{HugSchneiderConicalTessellations}. For completeness, we provide a full proof.
\end{remark}
\begin{remark}
By definition, the elements of the space $\mathcal K_{\text{con}}$  mentioned in Lemma~\ref{lemma:size_biased}  are closed convex (not necessarily polyhedral) cones in $\RR^{d+1}$ different from $\{0\}$. The distance between two such cones $C_1$ and $C_2$ is defined to be the Hausdorff distance between the spherically convex sets $C_1\cap \SS^d$ and $C_2\cap \SS^d$, where the sphere is endowed with the geodesic distance  {(alternatively, the angular Hausdorff distance as introduced in \cite{SchneiderConicSupport} can be used as well)}.
The space $\mathcal K_{\text{con}}$ can be identified with the space $\mathcal K_s$ of spherically convex sets  (to use the notation of~\cite[Section 6.5]{SW}) by identifying each closed convex cone $C\neq \{0\}$ with the spherically convex set $C\cap \SS^d$.
\end{remark}
\begin{proof}[Proof of Lemma~\ref{lemma:size_biased}]
Using first the fact that $u(S_n)$ is  {uniformly distributed} on $S_n\cap \SS^d$, we obtain
\begin{align*}
&\EE [f(O_{u(S_n)} S_n) C(n,d+1) \alpha(S_n)]\\
&=
\EE \left[ \frac 1 {\alpha(S_n)} \int_{S_n\cap \SS^d} f(O_v S_n)C(n,d+1) \alpha(S_n)\, \sigma_d(\dint v)\right]\\
&=
\int_{\SS^d}\EE \left[f(O_v S_n) C(n,d+1) \ind\{v\in S_n\} \right] \,\sigma_d(\dint v).
\end{align*}
Recalling that $S_n$ is chosen uniformly from the collection $\cones(H_1,\ldots,H_n)$ which  {almost surely} contains $C(n,d+1)$ elements, we can proceed as follows:
\begin{align*}
\lefteqn{\EE [f(O_{u(S_n)} S_n) C(n,d+1) \alpha(S_n)]}\\
&=
\int_{\SS^d} \Big(\int_{\cH^n} \sum_{C\in \cones(H_1,\ldots,H_n)} f(O_v C) \ind\{v\in S_n\} \,\nu_d^n(\dint(H_1,\ldots,H_n))\Big) \, \sigma_d(\dint v)\\
&=
\int_{\SS^d} \Big(\int_{\cH^n}  f(O_v C^{v}) \,\nu_d^n(\dint(H_1,\ldots,H_n))\Big)\, \sigma_d(\dint v),
\end{align*}
where $C^{v}$ denotes the almost surely unique cone from the collection $\cones(H_1,\ldots,H_n)$ which contains the point $v$. By rotational invariance of $H_1,\ldots,H_n$, the inner integral does not depend on the choice of $v\in \SS^d$. So, we may choose $v:= -e$, which leads to
\begin{align*}
\EE [f(O_{u(S_n)} S_n) C(n,d+1) \alpha(S_n)]
=
\int_{\cH^n}  f(O_{-e} C^{-e}) \,\nu_d^n(\dint(H_1,\ldots,H_n))
=
\EE [f(S_n^{-e})],
\end{align*}
since the cone $S_n^{-e}$ has the same law as $C^{-e}$ and is rotationally invariant.
\end{proof}

In the language of densities, Lemma~\ref{lemma:size_biased} can be restated as follows: the laws of the random cones $S_n^{-e}$ and $O_{u(S_n)}S_n$ are mutually absolutely continuous probability measures on the space $\mathcal K_{\text{con}}$ and
\begin{equation}\label{eq:radon_nik}
\frac{\dint \mu_{S_n^{-e}}}{\dint \mu_{O_{u(S_n)}S_n}} (C) = C(n,d+1) \alpha(C), \qquad C\in \mathcal K_{\text{con}}.
\end{equation}

\subsection{Profiles of cones}
To state results about weak convergence of random cones, we need to introduce the language of profiles (or cross-sections) of cones. As $n\to\infty$, the Cover-Efron random cone and the random cone $R_n$ become ``thick'' (i.e., close to a half-space), whereas the Schl\"afli random cone $S_n$ and the random cone $S_n^{-e}$ become ``thin'' (i.e., close to a ray). In order to have weak convergence, we need to rescale the cones, which is most conveniently done after replacing the cones by their profiles, defined as follows.

For a cone $C\subset \RR^{d+1}$ we denote by $u(C)$ a random point sampled uniformly from $C\cap \SS^d$. We denote by $\Tan_v$ the tangent space at $v\in\SS^d$ of $\SS^d$ and fix for every point $v\in\SS^d$ an isometry $I_v:\Tan_v\to\RR^d$ satisfying $I_v(v)=0$ and such that the map $(v_1,v_2)\mapsto I_{v_1} v_2$ (defined on the tangent bundle of $\SS^d$) is Borel-measurable. The \textbf{profile} of a (possibly random) cone $C\subset\RR^{d+1}$ can then be defined as $I_{u(C)}(C\cap\Tan_{u(C)})$. Excluding the cases when $C$ is empty, equal to $\RR^{d+1}$ or a half-space, the profile is a polytope in $\RR^d$.

We shall be interested in the following rescaled profiles:
\begin{equation}\label{eq:def_P_n_Q_n}
P_n:=n\,I_{-e}(S_n^{-e} \cap\Tan_{-e})\subset \RR^d
\qquad\text{and}\qquad
Q_n:=n\,I_{u(S_n)}(S_n\cap\Tan_{u(S_n)})\subset \RR^d.
\end{equation}
Both, $P_n$ and $Q_n$, are random convex closed subsets of $\RR^d$, but they need not be bounded. For example, the cone $S_n^{-e}$ need not be contained in the lower half-space. Fortunately, the probability of these events goes to $0$, as  the following lemma states.
\begin{lemma}\label{lem:compact_with_probab1}
We have
\begin{equation}\label{eq:limits_P_n_Q_n}
\lim_{n\to\infty} \PP[P_n\in \cK^d] = 1
\qquad
\text{ and }
\qquad
\lim_{n\to\infty} \PP[Q_n\in \cK^d] = 1.
\end{equation}
\end{lemma}
\begin{proof}
By~\eqref{eq:DualityRnAndSne}, to prove the first identity it suffices to show that the probability that $e$ is contained in the interior of $R_n$ converges to $1$, as $n\to\infty$. Pick $d+1$ points on $\SS_e^d$ such that $e$ is contained in the interior of their positive hull. If $\varepsilon>0$ is sufficiently small, the positive hull of the $\varepsilon$-perturbed points still contains $e$ in its interior. Now, the probability that each geodesic ball of radius $\varepsilon$ around the $d+1$ points contains at least one point from the collection $X_1,\ldots,X_n$ converges to $1$, as $n\to\infty$. On this event, the cone $R_n = \pos(X_1,\ldots,X_n)$ contains $e$ in its interior, hence the polar cone $S_n^{-e}$ is contained in the open lower half-space and $P_n\in \cK^d$. This proves the first claim in~\eqref{eq:limits_P_n_Q_n}.

To prove the second claim in~\eqref{eq:limits_P_n_Q_n},  we observe that the definition of $Q_n$ and the rotational invariance imply that
$$
Q_n = n\,I_{u(S_n)}(S_n\cap\Tan_{u(S_n)}) \eqdistr n\, I_{-e} (O_{u(S_n)} (S_n) \cap \Tan_{-e}),
$$
 {where we write $\overset{d}{=}$ for equality in distribution}.
Thus, our task reduces to showing that
$$
\lim_{n\to\infty} \PP[I_{-e} (O_{u(S_n)} (S_n) \cap \Tan_{-e})\in\cK^d]  = 1.
$$
To this end, it suffices to demonstrate that
$$
\lim_{n\to\infty} \PP[S_n \text{ contains a pair of unit vectors with angle $\geq \pi/2$}] = 0.
$$
To prove this, it suffices to show that the maximal diameter (in the sense of the geodesic distance on $\SS^d$) of a cell in the  {great hypersphere tessellation generated by the intersection of $\SS^d$ with $H_1,\ldots,H_n$} converges to $0$ in probability. However, this is true even almost surely, because the maximal diameter does not increase with $n$ and each cell will be eventually split into cells of arbitrarily small diameter with probability $1$ after adding sufficiently many new hyperplanes.
\end{proof}

Let $P_n^*$  be the random closed compact set whose distribution is the distribution of $P_n$  conditioned on the event $\{P_n\in \cK^d\}$. Similarly, let $Q_n^*$ be the random closed compact set distributed as $Q_n$ conditioned on the event $\{Q_n\in \cK^d\}$. We can view $P_n^*$ (respectively, $Q_n^*$) as the \textbf{rescaled profiles} of the random cone $S_n^{-e}=R_n^\circ$ (respectively, the Schl\"afli random cone $S_n$). In Section~\ref{sec:ConesAndProof}, we will prove the weak convergence of $P_n^*$ and $Q_n^*$ on the space $(\cK^d,\tau_H^d)$, as $n\to\infty$. Note that instead of conditioning on the events $\{P_n\in \cK^d\}$ and $\{Q_n\in \cK^d\}$ in the definition of $P_n^*$ and $Q_n^*$ it would alternatively be possible to re-define $P_n$ and $Q_n$ to be an arbitrary compact convex set on the events $\{P_n\notin \cK^d\}$, respectively $\{Q_n\notin \cK^d\}$. Since the probability of these events goes to $0$ as $n\to\infty$, this would not influence the weak convergence.

\section{Weak convergence of the profiles of random cones}\label{sec:ConesAndProof}

In this section we present the proof of Theorem \ref{thm:Intro} about the weak convergence of the profiles $Q_n$ of  Schl\"afli random cones $S_n$.

\subsection{Weak convergence for cones in the half-sphere}

We start by considering the random cones $R_n$ defined in~\eqref{eq:R_n_def}. These become ``thick'' in the large $n$ limit. The following result is already known from \cite{MarynychKabluchkoTemesvariThaele}, but we present here a new and short proof in order to demonstrate the versatility of our new method and since parts of the proof will be essential for what follows.

\begin{proposition}\label{prop:WeakConvRn}
Let $\Pi$ be a Poisson point process on $\RR^d\setminus\{0\}$ whose intensity measure has density $g_d(y)={2\over\omega_{d+1}}{1\over\|y\|^{d+1}}$  with respect to the Lebesgue measure on $\RR^d\setminus\{0\}$. Then,
$$
n^{-1}I_e(R_n\cap\Tan_{e})\to \mathrm{conv}(\Pi), \quad \text{ as } n\to\infty,
$$
weakly on $(\cK^d,\tau_H^d)$.
\end{proposition}
\begin{proof}
It is known~\cite{MarynychKabluchkoTemesvariThaele} that $I_e(R_n\cap\Tan_{e})$ is a beta' polytope  with parameter $\beta={d+1\over 2}$ in $\RR^d \equiv I_e(\Tan_{e})$, which is by definition the convex hull of random points $Z_1,\ldots,Z_n$  {in $\RR^d$} that are independent and have the $d$-dimensional Cauchy density
\begin{equation}\label{eq:beta_prime_density}
f_{d+1\over 2}(x) = {2\over\omega_{d+1}}{1\over(1+\|x\|^2)^{d+1\over 2}},\qquad x\in\RR^d.
\end{equation}
Here and in the following we use the standard notation $\kappa_d$ and $\omega_d$ for the volume of the unit ball in $\RR^d$ and the surface area of the unit sphere $\SS^{d-1}$, namely
$$
\kappa_d:={\pi^{d\over 2}\over\Gamma(1+{d\over 2})}
\qquad\text{and}\qquad
\omega_d := d\kappa_d = {2\pi^{d\over 2}\over\Gamma({d\over 2})}.
$$
In fact, the Cauchy distribution (which is a special case of the beta' distribution with parameter $\beta= \frac{d+1} 2$) appears as the image measure of the uniform distribution on $\SS_e^d$ under the gnomonic projection from $\SS_e^d$ to $\Tan_{e}$.
Beta' polytopes  were intensively studied in~\cite{kabluchko_poisson_zero,KabluchkoFormula,KabluchkoTemesvariThaele,KabluchkoThaele_VoronoiSphere,KabluchkoThaeleZaporozhets}, for example. Our aim is thus to show that the random polytopes
$$
T_n:= \frac 1n \mathrm{conv}(\{Z_1,\ldots,Z_n\}) \eqdistr I_e(R_n\cap\Tan_{e})
$$
converge to $\mathrm{conv} (\Pi)$ weakly on $(\cK^d,\tau_H^d)$. To this end, we use the method developed in Section~\ref{sec:weak_conv}.
For every $n\geq d+1$ the density $\varphi_n$ of $\iota^{-1}(T_n)$ with respect to $\mu_\infty^d$ is given by
\begin{align}\label{eq:PhiNdensity}
\varphi_n({\bf x}) = {n\choose m}\,m!\bigg(\prod_{i=1}^m n^df_{d+1\over 2}(nx_i)\bigg)\bigg(\int_{\mathrm{conv}(\{x_1,\ldots,x_m\})} n^d\,f_{d+1\over 2}(ny)\,\dint y\bigg)^{n-m},
\end{align}
where ${\bf x}=(x_1,\ldots,x_m)\in \widetilde \cP_m^d$ and $m\geq d+1$.  Here, the factor $n\choose m$ reflects the choice of the $m$ points that become the vertices, $m!$ takes into account the permutations of the $m$ vertices, and the last factor in the formula is the probability that the remaining $n-m$ points are inside the convex hull of the $m$ vertices. Taking into account~\eqref{eq:beta_prime_density}, we have that
\begin{align*}
\varphi_n({\bf x}) &= {n!\over (n-m)!}\,\Big({2\over\omega_{d+1}}\Big)^m\bigg(\prod_{i=1}^m{n^d\over(1+n^2\|x_i\|^2)^{d+1\over 2}}\bigg)\\
&\qquad\qquad\times\bigg(1-{2\over\omega_{d+1}}\int_{\RR^d\setminus \mathrm{conv}(\{x_1,\ldots,x_m\})}{n^d\over(1+n^2\|y\|^2)^{d+1\over 2}}\dint y\bigg)^{n-m}\\
&={n!\over (n-m)!}\,{1\over n^m}\Big({2\over\omega_{d+1}}\Big)^m\bigg(\prod_{i=1}^m{1\over({1\over n^2}+\|x_i\|^2)^{d+1\over 2}}\bigg)\\
&\qquad\qquad\times\bigg(1-{1\over n}{2\over\omega_{d+1}}\int_{\RR^d\setminus \mathrm{conv}(\{x_1,\ldots,x_m\})}{1\over({1\over n^2}+\|y\|^2)^{d+1\over 2}}\dint y\bigg)^{n-m}.
\end{align*}
Thus, for each ${\bf x}=(x_1,\ldots,x_m)\in \widetilde \cP_m^d$ we obtain
\begin{align*}
\lim_{n\to\infty}\varphi_n({\bf x}) = \Big({2\over\omega_{d+1}}\Big)^m\bigg(\prod_{i=1}^m{1\over \|x_i\|^{d+1}}\bigg)\exp\bigg(-{2\over\omega_{d+1}}\int_{\RR^d\setminus \mathrm{conv}(\{x_1,\ldots,x_m\})}{\dint y\over\|y\|^{d+1}}\bigg).
\end{align*}
Next, we compute the density $\varphi={\dint\mu_{T^\iota}\over\dint\mu_\infty^d}$, where $T$ is the convex hull of the Poisson point process $\Pi$ on $\RR^d\setminus\{0\}$ whose intensity measure has density $g_d(y)={2\over\omega_{d+1}}{1\over\|y\|^{d+1}}$ with respect to the Lebesgue measure. Recall that $\mu_\infty^d$ is a measure on $\widetilde \cP_\infty^d$ with the property that its restriction to $\widetilde\cP_m^d$, $m\geq d+1$, coincides with the Lebesgue measure. Since the intensity measure of $\Pi$ has density $g_d$, we have that, for all ${\bf x}=(x_1,\ldots,x_m)\in\widetilde\cP_m^d$,
\begin{align}\label{eq:PhiDef}
\nonumber\varphi({\bf x})
&=\bigg(\prod_{i=1}^mg_d(x_i)\bigg)\PP[\Pi(\RR^d\setminus\mathrm{conv}(\{x_1,\ldots,x_m\}))=0]\\
\nonumber&=\bigg(\prod_{i=1}^mg_d(x_i)\bigg)\exp\Big(-\int_{\RR^d\setminus\mathrm{conv}(\{x_1,\ldots,x_m\})}g_d(y)\,\dint y\Big)\\
&=\Big({2\over\omega_{d+1}}\Big)^m\bigg(\prod_{i=1}^m{1\over \|x_i\|^{d+1}}\bigg)\exp\bigg(-{2\over\omega_{d+1}}\int_{\RR^d\setminus\mathrm{conv}(\{x_1,\ldots,x_m\})}{\dint y\over\|y\|^{d+1}}\bigg).
\end{align}
This proves that $\varphi_n\to\varphi$ pointwise on $\widetilde \cP_\infty^d$ {, as $n\to\infty$}. Finally, we notice that $\PP[T_n\in \cP_\infty^d]=\PP[T\in \cP_\infty^d]=1$, because the probability that two vertices of $T_n$ or $T$ have the same first coordinate is $0$. We can thus apply Proposition \ref{prop:WeakCovAbstract} to conclude the result.
\end{proof}

By passing to the polar cone $R_n^\circ$, which has the same law as $S_n^{-e}$ by~\eqref{eq:DualityRnAndSne} and which becomes ``thin'' as $n\to\infty$, we arrive at the following result.

\begin{corollary}\label{cor:RnPolar}
Let $\Pi$ be a Poisson point process on $\RR^d\setminus\{0\}$ whose intensity measure has density $g_d(y)={2\over\omega_{d+1}}{1\over\|y\|^{d+1}}$ with respect to the Lebesgue measure. Then,
$$
P_n^*
\to \mathrm{conv}(\Pi)^\circ, \quad \text{ as } n\to\infty,
$$
weakly on $(\cK^d,\tau_H^d)$. Here, $\mathrm{conv}(\Pi)^\circ$ denotes the dual polytope of $\mathrm{conv}(\Pi)$.
\end{corollary}
\begin{proof}
Recall that $P_n^*$ has the distribution of $P_n = n\,I_{-e}(S_n^{-e} \cap\Tan_{-e})$ conditioned on the event $\{P_n\in \cK^d\}$. On this event, whose probability converges to $1$, as $n\to\infty$, $P_n$ is the dual polytope of $n^{-1}I_e(R_n\cap\Tan_{e})$, and the claim follows directly from Proposition~\ref{prop:WeakConvRn}, the continuity of the polarity map on the set $\cK_o^d$ of convex compact sets containing the origin in their interior (see \cite[Theorem 13.3.4]{Moszynska}), the continuous mapping theorem and the fact that $\PP[n^{-1}I_{e}(R_n\cap\Tan_{e})\in\cK_o^d]\to 1$, as $n\to\infty$,  {and the fact} that $\PP[\mathrm{conv}(\Pi)\in\cK_o^d]=1$. For the latter claim, see~\cite[Corollary~4.2]{MarynychKabluchkoTemesvariThaele}.
\end{proof}

\begin{remark}\label{rem:density}
In the sequel, we shall need the following statement, which is slightly stronger than Corollary~\ref{cor:RnPolar}:
$$
\lim_{n\to\infty} \frac {\dint (\mu_{\iota^{-1}(P_n^*)})} {\dint \mu_\infty^d}
=
\frac{\dint (\mu_{\iota^{-1}(\mathrm{conv}(\Pi)^\circ)})}{\dint \mu_\infty^d},
\qquad
\text{ $\mu_\infty^d$-a.e.\ on $\widetilde \cP_\infty^d$.}
$$
This pointwise convergence of densities can be demonstrated as follows. In the proof of Proposition~\ref{prop:WeakConvRn} we have shown that the densities of $\iota^{-1}(T_n)$ with respect to the measure $\mu_\infty^d$ converge almost everywhere on $\widetilde \cP_\infty^d$ to the density of $\iota^{-1}(\mathrm{conv} (\Pi))$ with respect to the same measure:
\begin{equation}\label{eq:tech0}
\lim_{n\to\infty} \frac {\dint (\mu_{\iota^{-1}(T_n)})} {\dint \mu_\infty^d}
=
\frac{\dint (\mu_{\iota^{-1}(\mathrm{conv} (\Pi))})}{\dint \mu_\infty^d},
\qquad
\text{ $\mu_\infty^d$-a.e.\ on $\widetilde \cP_\infty^d$.}
\end{equation}
Now, $P_n^*$ has the same distribution as the convex dual $T_n^\circ$ of $T_n$ conditioned on the random event that $T_n$ contains the origin in its interior:
\begin{equation}\label{eq:tech1}
\frac {\dint (\mu_{\iota^{-1}(P_n^*)})} {\dint \mu_\infty^d}
=
\frac{\dint (\mu_{\iota^{-1}(T_n^\circ)|\{0\in {\rm int}\, T_n\}})}{\dint \mu_\infty^d}.
\end{equation}
We claim that the density of $\iota^{-1}(T_n^\circ)$ conditioned on $\{0\in {\rm int}\, T_n\}$ can be obtained from the density of $T_n$ as follows:\\
\begin{equation}\label{eq:transform1}
\frac{\dint (\mu_{\iota^{-1}(T_n^\circ)|\{0\in {\rm int}\, T_n\}})}{\dint \mu_\infty^d} (\iota^{-1}(p^\circ))
=
\frac{\dint (\mu_{\iota^{-1}(T_n)})}{\dint \mu_\infty^d}
(\iota^{-1}(p))\cdot J(\iota^{-1}(p))\cdot \frac {\ind_{\{0\in {\rm int}\, p\}}} {\PP[0\in {\rm int}\, T_n]},
\end{equation}
for some (non-explicit) function $J$ {, the Jacobian of the polarity map,} and $\mu_\infty^d$-a.e.\ $x = \iota^{-1}(p)\in \widetilde \cP_\infty^d$.  Indeed, let $p\in \tilde P_\infty^d$ be some simplicial polytope containing $0$ in its interior. We think of $p$ as of a potential realization of $T_n$. The vertices of the dual polytope $p^\circ$ correspond to the facets of $p$ and the coordinates of the vertices are infinitely differentiable functions of the coordinates of vertices of $p$ in a neighborhood of $p$. Thus, on the level of coordinate representations in the space $\widetilde \cP_{\infty}^d$, the polarity map is a differentiable, one-to-one map in a neighborhood of $\iota^{-1}(p)$. Of course, we have to exclude the exceptional sets on which two vertices of the dual polytope have coinciding first coordinates, but these closed sets have $\mu_{\infty}^d$-measure zero. Using the transformation formula for the polarity map, we obtain~\eqref{eq:transform1}. Similar arguments, applied to $\mathrm{conv}(\Pi)$ and its dual, yield the formula
\begin{equation}\label{eq:transform2}
\frac{\dint (\mu_{\iota^{-1}((\mathrm{conv} (\Pi))^\circ)})}{\dint \mu_\infty^d} (\iota^{-1}(p^\circ))
=
\frac{\dint (\mu_{\iota^{-1}(\mathrm{conv} (\Pi))})}{\dint \mu_\infty^d}
(\iota^{-1}(p))\cdot J(\iota^{-1}(p))\cdot \ind_{\{0\in {\rm int}\, p\}},
\end{equation}
where $J$, the Jacobian of the polarity map, is the same as in~\eqref{eq:transform1}. Taking~\eqref{eq:tech0}, \eqref{eq:tech1}, \eqref{eq:transform1}, \eqref{eq:transform2} together, we obtain the almost everywhere density convergence stated in the remark.
\end{remark}

\subsection{Weak convergence of  Schl\"afli random cones}\label{subsec:weak_conv_main_proof}

Before presenting the proof of Theorem \ref{thm:Intro}, we recall the definition of the typical cell $Z$  {of a stationary and isotropic Poisson hyperplane process in $\RR^d$}. We consider a Poisson process $\eta$ on the space of affine hyperplanes in $\RR^d$ whose intensity measure is given by
$$
\xi(\,\cdot\,) := 2 \gamma\int_{\SS^{d-1}}\int_0^\infty \ind\{u^\perp + tu \in\,\cdot\,\}\,\dint t\,\sigma_{d-1}(\dint u),
$$
where we identified a hyperplane in $\RR^d$ with a pair $(u,t)\in\SS^{d-1}\times[0,\infty)$ representing the direction of a unit normal vector and the distance to the origin.  {Here, $\gamma$ stands for the constant given by~\eqref{eq:gamma} and $\sigma_{d-1}$ for the normalized spherical Lebesgue measure on $\SS^{d-1}$.} The hyperplanes of $\eta$ partition the space into almost surely countably many $d$-dimensional random polytopes. The collection of these polytopes is denoted by $\widehat{\eta}$  {and referred to as the stationary and isotropic Poisson hyperplane tessellation in $\RR^d$ with intensity $\gamma$.} Conditionally on $\eta$, we choose for each polytope $p\in \widehat{\eta}$ its ``center'' by picking one point $v(p)$ uniformly at random inside $p$. We now define a distribution $\mu_Z$ on the space $\cP^d$ as follows:
$$
\mu_Z(\,\cdot\,) := \Big(\EE\sum_{p\in\widehat{\eta}} \ind \{v(p)\in[0,1]^d\}\Big)^{-1}\EE\sum_{p\in\widehat{\eta}}\ind\{v(p)\in[0,1]^d,p-v(p)\in\,\cdot\,\}.
$$
A random polytope $Z$ distributed according to $\mu_Z$ is called the \textbf{typical cell} of $\eta$ centered at a uniform random point.
Note that usually  {in stochastic geometry} one centers the cells in a different way by using instead of the random center $v(P)$ some deterministic center function $c(p)$, for example the center of the smallest ball that can be circumscribed around $p$; see \cite{SW}. Intuitively, one might think of $Z$ as a random polytope ``uniformly'' selected  from the set of all cells in $\widehat{\eta}$ and centered at a uniform random point inside itself. Two realizations of Poisson hyperplane tessellations are shown in Figure \ref{fig2}.

\begin{figure}[t]
\begin{center}
\includegraphics[width=0.8\columnwidth]{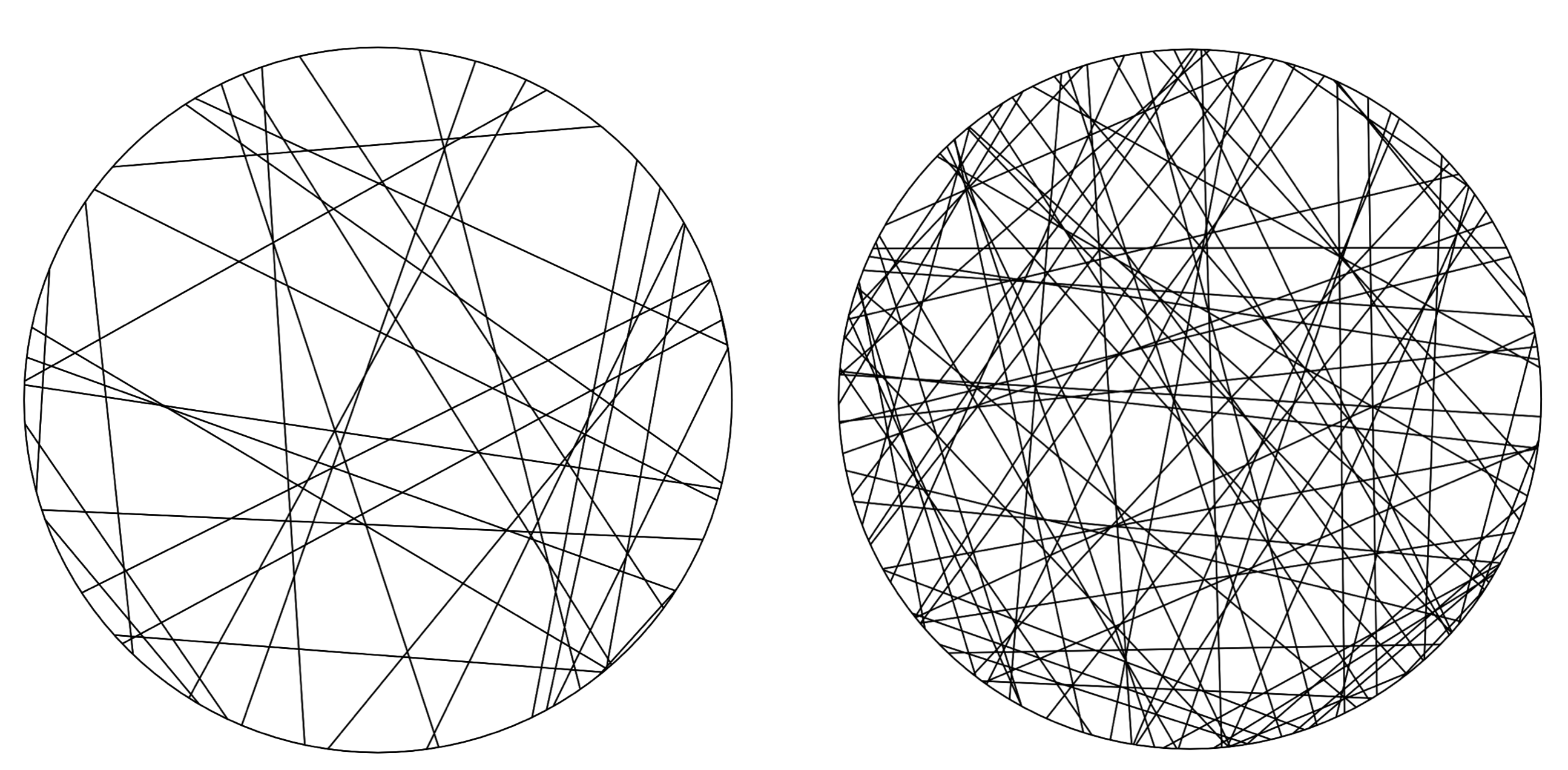}
\end{center}
\caption{Two realizations of stationary and isotropic Poisson hyperplane tessellations in $\RR^2$ with intensities $5$ (left panel) and $15$ (right panel) seen in a disc of radius $1$.}
\label{fig2}
\end{figure}

We are now prepared to restate and prove Theorem \ref{thm:Intro} about the weak convergence of the profiles $Q_n^*$ of the  Schl\"afli random cones $S_n$. Recall that $Q_n^*$ has the same distribution as $Q_n=n\,I_{u(S_n)}(S_n\cap\Tan_{u(S_n)})$ conditioned on the event $\{Q_n\in \cK^d\}$.
\begin{theorem}\label{thm:Intro_restate}
As $n\to\infty$, $Q_n^*$ converges to $Z$ weakly on  $(\cK^d,\tau_H^d)$.
\end{theorem}
\begin{proof}
Let $S_n$ be the  Schl\"afli random cone.
By rotational invariance, we have
$$
Q_n = n\,I_{u(S_n)}(S_n\cap\Tan_{u(S_n)}) \eqdistr n\, I_{-e} (O_{u(S_n)} (S_n) \cap \Tan_{-e}).
$$
Define the following set of cones:
$$
\cK_{\text{con}}^* = \{C\in \cK_{\text{con}}: \; -e\in C,\; I_{-e}(C\cap \Tan_{-e}) \in \cK^d\}.
$$
Define also the following  maps assigning to each cone in $\cK_{\text{con}}$ its rescaled profile:
$$
\Psi_n: \cK_{\text{con}}^* \to \cK^d,\qquad  C\mapsto  n\, I_{-e} (C\cap \Tan_{-e}).
$$
By Lemma~\ref{lem:compact_with_probab1}, the cone $O_{u(S_n)} (S_n)$ belongs to $\cK_{\text{con}}^*$ outside an event of  probability converging to $0$. The restrictions of the probability distributions of $O_{u(S_n)} (S_n)$ and $S_n^{-e}$ to $\cK_{\text{con}}^*$ are subprobability measures on $\cK_{\text{con}}^*$ defined as follows:
$$
\mu^*_{S_n^{-e}}(\cdot) := \mu_{S_n^{-e}}(\cdot \cap \cK_{\text{con}}^*),
\qquad
\mu^*_{O_{u(S_n)} (S_n)}(\cdot) := \mu_{O_{u(S_n)} (S_n)}(\cdot \cap \cK_{\text{con}}^*).
$$
By restricting~\eqref{eq:radon_nik} to $\cK_{\text{con}}^*$ we obtain that these subprobability measures are mutually absolutely continuous with density
\begin{equation}\label{eq:radon_nik_restr}
\frac{\dint \mu^*_{S_n^{-e}}}{\dint \mu^*_{O_{u(S_n)}S_n}} (C) = C(n,d+1) \alpha(C), \qquad C\in \mathcal \cK_{\text{con}}^*.
\end{equation}
Applying the injective transformation $\Psi_n^{-1}$ to the measures appearing in~\eqref{eq:radon_nik_restr}, we deduce that the image measures are mutually absolutely continuous with density
$$
\frac{\dint (\mu^*_{S_n^{-e}}\circ \Psi_n^{-1})}{\dint (\mu^*_{O_{u(S_n)}S_n}\circ \Psi_n^{-1})} (K) = \frac 12 \, C(n,d+1)\,\PC_{d+1\over 2}(n^{-1}K),\qquad K\in\cK^d,
$$
where $\PC_{\frac{d+1}{2}}$ denotes the beta' probability content  (with $\beta=\frac {d+1}{2}$) of a measurable set $B\subset\RR^d$ defined by
$$
\PC_{d+1\over 2}(B) := {2\over\omega_{d+1}}\int_{B}{\dint x\over(1+\|x\|^2)^{d+1\over 2}}.
$$
Here, we used the fact that the angle of a cone $C \in \cK_{\text{con}}^*$ equals $1/2$ times the beta' probability content of its profile $I_{-e} (C\cap \Tan_{-e})$ (the factor $1/2$ comes from the fact that the solid angle of a half-space is $1/2$). Inverting the density and applying the map $\iota$, we obtain
$$
\frac{\dint (\mu^*_{O_{u(S_n)}S_n}\circ \Psi_n^{-1}\circ \iota)} {\dint (\mu^*_{S_n^{-e}}\circ \Psi_n^{-1}\circ \iota)} (x) = \frac{2}{C(n,d+1)\,\PC_{d+1\over 2}(n^{-1} \iota (x))},\qquad x\in  \widetilde \cP_\infty^d.
$$
Passing to densities with respect to the ``Lebesgue measure'' $\mu_\infty^d$ on $\widetilde \cP_\infty^d$, we arrive at
$$
\frac{\dint (\mu^*_{O_{u(S_n)}S_n}\circ \Psi_n^{-1}\circ \iota)} {\dint \mu_\infty^d} (x)
=
\frac{2}{C(n,d+1)\,\PC_{d+1\over 2}(n^{-1} \iota (x))}\cdot \frac {\dint (\mu^*_{S_n^{-e}}\circ \Psi_n^{-1}\circ \iota)} {\dint \mu_\infty^d} ,\quad x\in  \widetilde \cP_\infty^d.
$$
Finally, recalling that the probability distribution of $Q_n^*$ is defined to be
$$
\mu_{Q_n^*}(\cdot)
=
\frac{\PP[ n\, I_{-e} (O_{u(S_n)} (S_n) \cap \Tan_{-e}) \in \cdot, O_{u(S_n)} (S_n) \in \cK_{\text{con}}^*]}{\PP[O_{u(S_n)} (S_n) \in \cK_{\text{con}}^*]}
=
\frac{\mu^*_{O_{u(S_n)}S_n} \circ \Psi_n^{-1}(\cdot)}{\PP[Q_n \in \cK^d]},
$$
we arrive at
$$
\frac{\dint (\mu_{Q_n^*}\circ \iota)} {\dint \mu_\infty^d} (x)
=
\frac{2}{C(n,d+1)\,\PC_{d+1\over 2}(n^{-1} \iota (x))\, \PP[Q_n \in \cK^d]}\cdot \frac {\dint (\mu^*_{S_n^{-e}}\circ \Psi_n^{-1}\circ \iota)} {\dint \mu_\infty^d} ,\quad x\in  \widetilde \cP_\infty^d.
$$
We would like to apply Proposition~\ref{prop:WeakCovAbstract} to the random polytopes $Q_n^*$. To this end, we need to show that the density on the right-hand side has an almost sure limit, as $n\to\infty$. First we recall from Remark~\ref{rem:density}  that
$$
\lim_{n\to\infty} \frac {\dint (\mu^*_{S_n^{-e}}\circ \Psi_n^{-1}\circ \iota)} {\dint \mu_\infty^d}
=
\frac{\dint (\mu_{\mathrm{conv}(\Pi)^\circ} \circ \iota)}{\dint \mu_\infty^d} \qquad \text{ $\mu_\infty^d$-a.e.\ on $\widetilde \cP_\infty^d$},
$$
where $\mathrm{conv} (\Pi)^\circ$ is the convex dual of $\mathrm{conv}(\Pi)$.
By Lemma~\ref{lem:compact_with_probab1}, we have $\lim_{n\to\infty}\PP[Q_n\in \cK^d] = 1$.
Then, note that by~\eqref{eq:C_n_d+1},
$$
\lim_{n\to\infty} {C(n,d+1)\over {2\over d!}\,n^d} = 1
$$
Moreover, using the substitution $u=n^{-1}v$, we see that for every polytope $p\in \cP^d$,
\begin{align*}
\PC_{d+1\over 2}(n^{-1}p) &= {2\over\omega_{d+1}}\int_{n^{-1}p}{\dint u\over(1+\|u\|^2)^{d+1\over 2}}\\
&= {2\over\omega_{d+1}}\,n^{-d}\,\int_p {\dint v\over\big(1+\|n^{-1}v\|^2\big)^{d+1\over 2}}
\end{align*}
and hence, by monotone convergence,
\begin{align*}
\lim_{n\to\infty} n^d\PC_{d+1\over 2}(n^{-1}\iota(x)) = {2\over\omega_{d+1}}\,\vol(\iota(x)),
\end{align*}
where $\vol(p)$ denotes the $d$-dimensional Lebesgue measure of $p$.
Taking everything together, we obtain
\begin{equation}\label{eq:lim_density}
\lim_{n\to\infty}
\frac{\dint (\mu_{Q_n^*}\circ \iota)} {\dint \mu_\infty^d} (x)
=
\frac{\dint \mu_{\iota^{-1}(\mathrm{conv}(\Pi)^\circ)}}{\dint \mu_\infty^d}(x)
\cdot
{d!\,\omega_{d+1}\over 2\,\vol(\iota(x))}
\qquad \text{ $\mu_\infty^d$-a.e.\ on $\widetilde \cP_\infty^d$.}
\end{equation}
By Proposition~\ref{prop:WeakCovAbstract}, the random polytope $Q_n^*$  converges weakly on  $(\cK^d,\tau_H^d)$ to a random polytope with density given by the right-hand side of~\eqref{eq:lim_density}.

It remains to identify the random polytope appearing in the limit.  From~\cite[Section~1.6]{KabluchkoThaeleZaporozhets} we know that $\mu_{\mathrm{conv}(\Pi)^\circ}=\mu_{Z_0}$, where $Z_0$ stands for the \textbf{zero cell} of a stationary and isotropic Poisson hyperplane tessellation in $\RR^d$ with intensity
$$
\gamma
=
{\omega_d\over 2}{2\over\omega_{d+1}}
=
{1\over\sqrt{\pi}}{\Gamma({d+1\over 2})\over\Gamma({d\over 2})}
$$
 {(we recall that $Z_0$ is the almost surely uniquely determined cell of $\widehat{\eta}$ containing the origin in its interior).}
This value for $\gamma$ comes from Remark~1.24 in~\cite{KabluchkoThaeleZaporozhets}, since our limiting Poisson point process $\Pi$ has intensity function ${2\over\omega_{d+1}}\|x\|^{-d-1}$ on $\RR^d\backslash\{0\}$. According to~\cite[p.\ 490]{SW}, the expected volume of $Z$ is given by
\begin{align*}
\EE[\vol(Z)] &= \Big({\omega_d\over\kappa_{d-1}}\Big)^d{1\over\kappa_d\,\gamma^d} = {1\over\kappa_d}\Big({\omega_d\over\kappa_{d-1}}\Big)^d\Big({\omega_{d+1}\over\omega_d}\Big)^d = {1\over\kappa_d}\Big({\omega_{d+1}\over\kappa_{d-1}}\Big)^d  =: c_d.
\end{align*}
Moreover, by \cite[Theorem 10.4.1]{SW} the laws of the zero cell $Z_0$ and the typicall cell $Z$ are mutually absolutely continuous with the density
$$
{\dint\mu_{Z_0}\over\dint\mu_Z}(p) = {\vol(p)\over\EE[\vol(Z)]}, \qquad p\in \cK^d.
$$
Equivalently,
$$
{\dint\mu_{Z}\over\dint\mu_{Z_0}}(p) = {\EE[\vol(Z)]\over\vol(p)} = {2\,c_d\over d!\,\omega_{d+1}}\,{d!\,\omega_{d+1}\over 2\,\vol(p)} = {d!\,\omega_{d+1}\over 2\,\vol(p)}, \qquad p\in \cK^d,
$$
because
$$
{2\,c_d\over d!\,\omega_{d+1}} = {2\over d!\,\kappa_d\omega_{d+1}}\Big({\omega_{d+1}\over\kappa_{d-1}}\Big)^d = (2\pi)^{-d}\Bigg({2\pi^{d+1\over 2}\over\Gamma({d+1\over 2})}{\Gamma({d+1\over 2})\over\pi^{d-1\over 2}}\Bigg)^d = 1,
$$
where we used that $\kappa_d\omega_{d+1}={2^{d+1}\pi^d\over d!}$. This means that
\begin{align*}
\lim_{n\to\infty}
\frac{\dint \mu_{\iota^{-1}(Q_n^*)}} {\dint \mu_\infty^d} (x)
&=
\frac{\dint \mu_{\iota^{-1}(Z_0)}}{\dint \mu_\infty^d}(x) \cdot {d!\,\omega_{d+1}\over 2\,\vol(\iota(x))}\\
&=
\frac{\dint \mu_{\iota^{-1}(Z_0)}}{\dint \mu_\infty^d}(x) \cdot  {\dint\mu_{\iota^{-1}(Z)}\over\dint\mu_{\iota^{-1}(Z_0)}}(x) \\
&=
\frac{\dint\mu_{\iota^{-1}(Z)}}{\dint \mu_\infty^d}(x),
\qquad \text{ $\mu_\infty^d$-a.e.\ on $\widetilde \cP_\infty^d$.}
\end{align*}
Applying Proposition~\ref{prop:WeakCovAbstract} with $T_n=Q_n^*$ and $T=Z$ we thus conclude that $Q_n^* \to Z$, as $n\to\infty$, weakly in $(\cK^d,\tau_H^d)$.
\end{proof}

\subsection*{Acknowledgement}
Z.K.\ has been supported by the German Research Foundation under Germany's Excellence Strategy  EXC 2044 -- 390685587, Mathematics M\"unster: Dynamics - Geometry - Structure.

\addcontentsline{toc}{section}{References}


\end{document}